\documentclass[preprint,12pt]{elsarticle}

\usepackage{amssymb}
\usepackage{amsthm}

\newcommand{\sysn}{\left\{\begin{array}{rcl}}
\newcommand{\sysk}{\end{array}\right.}

\newtheorem{theorem}{Theorem}[section]
\newtheorem{lemma}[theorem]{Lemma}

\theoremstyle{example}

\newtheorem{proposition}[theorem]{Proposition}
\theoremstyle{definition}
\newtheorem{definition}[theorem]{Definition}

\journal{...}

\begin{document}

\begin{frontmatter}



\title{On selective sequential separability of function spaces with the compact-open topology}


\author{Alexander V. Osipov}


\ead{OAB@list.ru}


\address{Krasovskii Institute of Mathematics and Mechanics, Ural Federal
 University,

 Ural State University of Economics, 620219, Yekaterinburg, Russia}

\begin{abstract}
For a Tychonoff space $X$, we denote by $C_k(X)$ the space of all
real-valued continuous functions on $X$ with the compact-open
topology. A subset $A\subset X$ is said to be sequentially dense
in $X$ if every point of $X$ is the limit of a convergent sequence
in $A$.

A space $C_k(X)$ is selective sequentially separable (in
Scheepers' terminology: $C_k(X)$ satisfies
$S_{fin}(\mathcal{S},\mathcal{S})$) if whenever $(S_n : n\in
\mathbb{N})$ is a sequence of sequentially dense subsets of
$C_k(X)$, one can pick finite $F_n\subset S_n$ ($n\in \mathbb{N}$)
such that $\bigcup \{F_n: n\in \mathbb{N}\}$ is sequentially dense
in $C_k(X)$.

In this paper, we have gave characterization for $C_k(X)$ to
satisfy $S_{fin}(\mathcal{S},\mathcal{S})$.

\end{abstract}


\end{frontmatter}



\section{Introduction}
\label{}
 If $X$ is a space and $A\subseteq X$, then the sequential closure of $A$,
 denoted by $[A]_{seq}$, is the set of all limits of sequences
 from $A$. A set $D\subseteq X$ is said to be sequentially dense
 if $X=[D]_{seq}$. A space $X$ is called sequentially separable if
 it has a countable sequentially dense set.

Let $X$ be a topological space, and $x\in X$. A subset $A$ of $X$
{\it converges} to $x$, $x=\lim A$, if $A$ is infinite, $x\notin
A$, and for each neighborhood $U$ of $x$, $A\setminus U$ is
finite. Consider the following collection:

$\bullet$ $\Omega_x=\{A\subseteq X : x\in \overline{A}\setminus
A\}$;

$\bullet$ $\Gamma_x=\{A\subseteq X : x=\lim A\}$.

Note that if $A\in \Gamma_x$, then there exists $\{a_n\}\subset A$
converging to $x$. So, simply $\Gamma_x$ may be the set of
non-trivial convergent sequences to $x$.

Many topological properties are defined or characterized in terms
 of the following classical selection principles.
 Let $\mathcal{A}$ and $\mathcal{B}$ be sets consisting of
families of subsets of an infinite set $X$. Then:

$S_{1}(\mathcal{A},\mathcal{B})$ is the selection hypothesis: for
each sequence $(A_{n}: n\in \mathbb{N})$ of elements of
$\mathcal{A}$ there is a sequence $\{b_{n}\}_{n\in \mathbb{N}}$
such that for each $n$, $b_{n}\in A_{n}$, and $\{b_{n}:
n\in\mathbb{N} \}$ is an element of $\mathcal{B}$.

$S_{fin}(\mathcal{A},\mathcal{B})$ is the selection hypothesis:
for each sequence $(A_{n}: n\in \mathbb{N})$ of elements of
$\mathcal{A}$ there is a sequence $\{B_{n}\}_{n\in \mathbb{N}}$ of
finite sets such that for each $n$, $B_{n}\subseteq A_{n}$, and
$\bigcup_{n\in\mathbb{N}}B_{n}\in\mathcal{B}$.

\medskip

In this paper, by a cover we mean a nontrivial one, that is,
$\mathcal{U}$ is a cover of $X$ if $X=\bigcup \mathcal{U}$ and
$X\notin \mathcal{U}$.

An open cover $\mathcal{U}$ of a space $X$ is called:

$\bullet$  a $k$-cover if each compact subset $C$ of $X$ is
contained in an element of $\mathcal{U}$ and $X\notin \mathcal{U}$
(i.e. $\mathcal{U}$ is a non-trivial cover);

$\bullet$  a $\gamma_k$-cover if $\mathcal{U}$ is infinite,
$X\notin \mathcal{U}$, and for each compact subset $C$ of $X$ the
set $\{U\in \mathcal{U} : C\nsubseteq U\}$ is finite.

\medskip

For a Tychonoff space $X$, we denote by $C_k(X)$ the space of all
real-valued continuous functions on $X$ with the compact-open
topology. Subbase open sets of $C_k(X)$ are of the form $[A,
U]=\{f\in C(X): f(A)\subset U \}$, where $A$ is a compact subset
of $X$  and $U$ is a non-empty open subset of $\mathbb{R}$.
Sometimes we will write the basic neighborhood of the point $f$ as
$<f,A,\epsilon>$ where $<f,A,\epsilon>:=\{g\in C(X):
|f(x)-g(x)|<\epsilon$ $\forall x\in A\}$, $A$ is a compact subset
of $X$ and $\epsilon>0$.

For a topological space $X$ we denote:

$\bullet$ $\Gamma_k$ --- the family of open $\gamma_k$-covers of
$X$;

$\bullet$ $\mathcal{K}$ --- the family of open $k$-covers of $X$;

$\bullet$ $\mathcal{K}_{cz}^{\omega}$ --- the family of countable
co-zero $k$-covers of $X$;

$\bullet$ $\mathcal{D}$ --- the family of dense subsets of
$C_k(X)$;

$\bullet$ $\mathcal{S}$ --- the family of sequentially dense
subsets of $C_k(X)$;

$\bullet$ $\mathbb{K}(X)$ --- the family of all non-empty compact
subsets of $X$.

\bigskip
A space $X$ is said to be a $\gamma_k$-set if each $k$-cover
$\mathcal{U}$ of $X$ contains a countable set $\{U_n : n\in
\mathbb{N}\}$ which is a $\gamma_k$-cover of $X$ \cite{koc1}.

\section{Main definitions and notation}

$\bullet$ A space $X$ is $R$-separable, if $X$ satisfies
$S_1(\mathcal{D}, \mathcal{D})$ (Def. 47, \cite{bbm1}).

$\bullet$ A space $X$ is selective separability ($M$-separable),
if $X$ satisfies $S_{fin}(\mathcal{D}, \mathcal{D})$.

$\bullet$ A space $X$ is selectively sequentially separable
($M$-sequentially separable), if $X$ satisfies
$S_{fin}(\mathcal{S}, \mathcal{S})$ (Def. 1.2, \cite{bc}).

\medskip

For a topological space $X$ we have the next relations of
selectors for sequences of dense sets of $X$.

\medskip
\begin{center}
$S_1(\mathcal{S},\mathcal{S}) \Rightarrow
S_{fin}(\mathcal{S},\mathcal{S}) \Rightarrow
S_1(\mathcal{S},\mathcal{D}) \Rightarrow
S_{fin}(\mathcal{S},\mathcal{D})$ \\  \, \, $\Uparrow$ \, \, \, \,
\,\, \,  \, \, \, $ \Uparrow $ \,\, \, \, \,\, \, \,
\, \, \, $\Uparrow $ \,\, \, \, \,\, \, \,\, \, \, $\Uparrow$ \\
$S_1(\mathcal{D},\mathcal{S}) \Rightarrow
S_{fin}(\mathcal{D},\mathcal{S}) \Rightarrow
S_1(\mathcal{D},\mathcal{D}) \Rightarrow
S_{fin}(\mathcal{D},\mathcal{D})$

\end{center}

\medskip

\bigskip

We write $\Pi (\mathcal{A}_x, \mathcal{B}_x)$ without specifying
$x$, we mean $(\forall x) \Pi (\mathcal{A}_x, \mathcal{B}_x)$.

$\bullet$  A space $X$ has {\it property $\alpha_2$} ($\alpha_2$
in the sense of Arhangel'skii), if $X$ satisfies
$S_{1}(\Gamma_x,\Gamma_x)$ \cite{arh0}.

$\bullet$ A space $X$ has {\it property $\alpha_4$} ($\alpha_4$ in
the sense of Arhangel'skii), if $X$ satisfies
$S_{fin}(\Gamma_x,\Gamma_x)$ \cite{arh0}.

\bigskip

 So we have three types of topological properties
described through the selection principles:

$\bullet$  local properties of the form $S_*(\Phi_x,\Psi_x)$;

$\bullet$  global properties of the form $S_*(\Phi,\Psi)$;

$\bullet$  semi-local of the form $S_*(\Phi,\Psi_x)$.

\medskip

In a series of papers it was demonstrated that $\gamma$-covers,
Borel covers, $k$-covers play a key role in function spaces
(\cite{cmkm,koc,koc1,koc2,mkn,mkm,mc,os1, os2,os4,os5,pp,sak1} and
many others). We continue to investigate applications of
$k$-covers in function spaces with the compact-open topology.

A great attention has recently received the notions of selective
separability  and selectively sequentially separable
($S_{fin}(\mathcal{S},\mathcal{S})$) \cite{bbm1,bbm2,glm,gs}. In
this paper, we have gave characterizations for $C_k(X)$ to satisfy
$S_{fin}(\mathcal{S},\mathcal{S})$, $S_{fin}(\mathcal{S},\Gamma_{
x})$, and $S_{fin}(\Gamma_{x},\Gamma_{x})$.

\section{Main results}

\begin{definition}
A $\gamma_k$-cover $\mathcal{U}$ of co-zero sets of $X$  is {\bf
$\gamma_k$-shrinkable} if there exists a $\gamma_k$-cover $\{F(U)
: U\in \mathcal{U}\}$ of zero-sets of $X$ with $F(U)\subset U$ for
every $U\in \mathcal{U}$.
\end{definition}

\medskip

For a topological space $X$ we denote:

$\bullet$ $\Gamma^{sh}_k$ --- the family of $\gamma_k$-shrinkable
$\gamma_k$-covers of $X$.


\medskip

Similar to the proof that
$S_{1}(\mathcal{K},\Gamma_k)=S_{fin}(\mathcal{K},\Gamma_{k})$
(Theorem 5 in \cite{koc1}), we prove the following

\begin{lemma}\label{lem2} For a space $X$ the following are equivalent:

\begin{enumerate}

\item  $X$ satisfies $S_{fin}(\Gamma^{sh}_k,\Gamma_k)$;

\item  $X$ satisfies $S_{1}(\Gamma^{sh}_k,\Gamma_k)$.

\end{enumerate}

\end{lemma}

\begin{proof} $(1)\Rightarrow(2)$. Let $(\mathcal{U}_n : n\in
\mathbb{N})$ be a sequence of (countable) $\gamma_k$-shrinkable
$\gamma_k$-covers of $X$; suppose that for each $n\in \mathbb{N}$,
$\mathcal{U}_n=\{U_{n,m} : m\in \mathbb{N}\}$. For each $n\in
\mathbb{N}$, let $\mathcal{V}_n$ denote the family of sets of the
form $U_{1,k_1}\cap U_{2,k_2}\cap...\cap U_{n,k_n}$. Then
$(\mathcal{V}_n : n\in \mathbb{N})$ is a sequence of
$\gamma_k$-shrinkable $\gamma_k$-covers of $X$. Since $X$
satisfies $S_{fin}(\Gamma^{sh}_k,\Gamma_k)$ choose for each $n\in
\mathbb{N}$ a finite subset $\mathcal{W}_n$ of $\mathcal{V}_n$
such that $\bigcup_{n\in \mathbb{N}} \mathcal{W}_n$ is a
$\gamma_k$-cover of $X$. (Note that some $\mathcal{W}_n's$ can be
empty.)

As $\bigcup_{n\in \mathbb{N}} \mathcal{W}_n$ is infinite and all
$\mathcal{W}'s$ are finite, there exists a sequence
$m_1<m_2<...<m_p<...$ in $\mathbb{N}$ such that for each $i\in
\mathbb{N}$ we have $\mathcal{W}_{m_i}\setminus \bigcup_{j<i}
\mathcal{W}_{m_j}\neq \emptyset$. Choose an element $W_{m_i}\in
\mathcal{W}_{m_i}\setminus \bigcup_{j<i} \mathcal{W}_{m_j}$, $i\in
\mathbb{N}$, and fix its representation $W_{m_i}=U_{1,k_1}\cap
U_{2,k_2}\cap...\cap U_{m_i,k_{m_i}}$ as above.

Since each infinite subset of a $\gamma_k$-cover is also a
$\gamma_k$-cover, we have that the set $\{W_{m_i}: i\in
\mathbb{N}\}$ is a $\gamma_k$-cover of $X$. For each $n\leq m_1$
let $U_n\in \mathcal{U}_n$ be the $n$-th coordinate of $W_{m_1}$
in the chosen representation of $W_{m_1}$, and for each $n\in
(m_i, m_{i+1}]$, $i\geq 1$, let $U_n\in \mathcal{U}_n$ be the
$n$-th coordinate of $W_{m_{i+1}}$ in the above representation of
$W_{m_{i+1}}$. Observe that each $U_n\supset W_{m_{i+1}}$.
Therefore, we obtain a sequence $(U_n :n\in \mathbb{N})$ of
elements, one from each $\mathcal{U}_n$, which form a
$\gamma_k$-cover of $X$ and show that $X$ satisfies
$S_{1}(\Gamma^{sh}_k,\Gamma_k)$.
\end{proof}

The symbol ${\bf 0}$ denote the constantly zero function in
$C_k(X)$. Because $C_k(X)$ is homogeneous we can work with ${\bf
0}$ to study local and semi-local properties of $C_k(X)$.

\begin{theorem}\label{th07} For a Tychonoff space $X$  the following statements are
equivalent:

\begin{enumerate}

\item  $C_k(X)$ satisfies $S_{1}(\Gamma_{\bf 0},\Gamma_{\bf 0})$
[property $\alpha_2$];

\item  $X$ satisfies $S_{1}(\Gamma^{sh}_k,\Gamma_k)$.

\end{enumerate}

\end{theorem}

\begin{proof}  $(1)\Rightarrow(2)$. Let $(\mathcal{U}_n : n\in
\mathbb{N})$ be a sequence of (countable) $\gamma_k$-shrinkable
$\gamma_k$-covers of $X$; suppose that for each $n\in \mathbb{N}$,
$\mathcal{U}_n=\{U_{n,m} : m\in \mathbb{N}\}$ and $\{F(U_{n,m}) :
U_{n,m}\in \mathcal{U}_n\}$ is a $\gamma_k$-cover of zero-sets of
$X$ with $F(U_{n,m})\subset U_{n,m}$ for every $U_{n,m}\in
\mathcal{U}_n$. Consider $S_n=\{f_{n,m}\in C(X) :
f_{n,m}\upharpoonright F(U_{n,m})\equiv 0$,
$f_{n,m}\upharpoonright (X\setminus U_{n,m})\equiv 1$ for
$U_{n,m}\in \mathcal{U}_n\}$. Since $\{F(U_{n,m}) : U_{n,m}\in
\mathcal{U}_n\}$ is a $\gamma_k$-cover of $X$, then $S_n\in
\Gamma_{\bf 0}$. By (1), there is $\{f_{n,m(n)}: n\in
\mathbb{N}\}$ such that $f_{n,m(n)}\in S_n$ and $\{f_{n,m(n)}:
n\in \mathbb{N}\}\in \Gamma_{\bf 0}$. Claim that $\{U_{n,m(n)}:
n\in \mathbb{N}\}\in \Gamma_k$. Suppose $A\in \mathbb{K}(X)$ and
$W=[A,(-\frac{1}{2},\frac{1}{2})]$ is a base neighborhood of ${\bf
0}$ then there exist $n'\in \mathbb{N}$ such that $f_{n,m(n)}\in
W$ for every $n>n'$. It follows that $A\subset U_{n,m(n)}$ for
every $n>n'$.

$(2)\Rightarrow(1)$. Let $S_n\in \Gamma_{\bf 0}$ for every $n\in
\mathbb{N}$; suppose that for each $n\in \mathbb{N}$,
$S_n=\{f_{n,j}: j\in \mathbb{N}\}$.

 Consider
$\mathcal{V}_n=\{f_{n,j}^{-1}((-\frac{1}{n},\frac{1}{n})) :
f_{n,j}\in S_n\}$ for each $n\in \mathbb{N}$. Note that
$\mathcal{W}_n=\{f_{n,j}^{-1}([-\frac{1}{n+1},\frac{1}{n+1}]) :
f_{n,j}\in S_n\}$ is a $\gamma_k$-cover of zero-sets of $X$.
Hence, $\mathcal{V}_n\in \Gamma^{sh}_k$ for each $n\in
\mathbb{N}$. By (2), there is $\{f_{n,j(n)} : n\in \mathbb{N}\}$
such that $\{f_{n,j(n)}^{-1}((-\frac{1}{n},\frac{1}{n})) : n\in
\mathbb{N}\}\in \Gamma_k$. Claim that $\{f_{n,j(n)} : n\in
\mathbb{N}\}\in \Gamma_{\bf 0}$. Let $[A,(-\epsilon, \epsilon)]$
be a base neighborhood of ${\bf 0}$ where $A\in \mathbb{K}(X)$ and
$\epsilon>0$. There is $n'\in \mathbb{N}$ such that $A\subset
f_{n,j(n)}^{-1}((-\frac{1}{n},\frac{1}{n}))$ for each $n>n'$.
There is $n''>n'$ such that $\frac{1}{n''}<\epsilon$, hence,
$f_{n,j(n)}\in [A,(-\frac{1}{n''}, \frac{1}{n''})]\subset
[A,(-\epsilon, \epsilon)]$ for each $n>n''$.

\end{proof}

\begin{proposition}(Proposition 4.2 in \cite{bbm2}). Every
selective sequentially separable space is sequentially separable.
\end{proposition}

We shall prove the following theorem under the condition that the
space $C_k(X)$ is sequentially separable. Denote this condition
for the space $X$ as $X\models SS$.

\begin{theorem}\label{th7} For a Tychonoff space $X$ such that $X\models SS$  the following statements are
equivalent:

\begin{enumerate}

\item  $C_k(X)$ satisfies $S_{1}(\mathcal{S},\mathcal{S})$;

\item  $C_k(X)$ satisfies $S_{1}(\mathcal{S},\Gamma_{\bf 0})$;

\item  $C_k(X)$ satisfies $S_{1}(\Gamma_{\bf 0},\Gamma_{\bf 0})$
[property $\alpha_2$];

\item  $X$ satisfies $S_{1}(\Gamma^{sh}_k,\Gamma_k)$;

\item $C_k(X)$ satisfies $S_{fin}(\mathcal{S},\mathcal{S})$
[selective sequentially separable];

\item  $C_k(X)$ satisfies $S_{fin}(\mathcal{S},\Gamma_{\bf 0})$;

\item  $C_k(X)$ satisfies $S_{fin}(\Gamma_{\bf 0},\Gamma_{\bf 0})$
[property $\alpha_4$];

\item  $X$ satisfies $S_{fin}(\Gamma^{sh}_{k},\Gamma_{k})$.

\end{enumerate}

\end{theorem}

\begin{proof} $(1)\Rightarrow(4)$. Let $\{\mathcal{U}_i\}\subset \Gamma^{sh}_k$, $\mathcal{U}_i=\{ U^{m}_i :m\in \mathbb{N}\}$ for each
$i\in \mathbb{N}$ and $S=\{h_m: m\in \mathbb{N}\}$ be a countable
sequentially dense subset of $C_k(X)$. For each $i\in \mathbb{N}$
we consider a countable sequentially dense subset $S_i$ of
$C_k(X)$ where $S_i=\{ f^m_i\in C(X) : f^m_i\upharpoonright
F(U^{m}_i)=h_m$ and $f^m_i\upharpoonright (X\setminus U^{m}_i)=1$
for $m \in \mathbb{N} \}$. Since $S$ is a countable sequentially
dense subset of $C_k(X)$, we have that $S_i$ is a countable
sequentially dense subset of $C_k(X)$ for each $i\in \mathbb{N}$.
Let $h\in C(X)$, there is a set $\{h_{m_s}: s\in
\mathbb{N}\}\subset S$ such that $\{h_{m_s}\}_{s\in \mathbb{N}}$
converge to $h$.
 Let $K$
be a compact subset of $X$, $\epsilon>0$ and $W=<h, K,\epsilon>$
be a base neighborhood of $h$, then there is a number $m_0$ such
that $K\subset F(U^{m}_i)$ for $m>m_0$ and $h_{m_s}\in W$ for
$m_s>m_0$. Since $f^{m_s}_i\upharpoonright K=
h_{m_s}\upharpoonright K$ for each $m_s>m_0$, $f^{m_s}_i\in W$ for
each $m_s>m_0$. It follows that a sequence $\{f^{m_s}_i\}_{s\in
\mathbb{N}}$ converge to $h$.

Since  $C_k(X)$ satisfies $S_{1}(\mathcal{S},\mathcal{S})$, there
is a sequence $\{f^{m(i)}_{i}\}_{i\in\mathbb{N}}$ such that for
each $i$, $f^{m(i)}_{i}\in S_i$, and $\{f^{m(i)}_{i}:
i\in\mathbb{N} \}$ is an element of $\mathcal{S}$.

Consider a set $\{U^{m(i)}_{i}: i\in \mathbb{N}\}$.

(a). $U^{m(i)}_{i}\in \mathcal{U}_{i}$.

(b). $\{U^{m(i)}_{i}: i\in \mathbb{N}\}$ is a $\gamma_k$-cover of
$X$.

There is a sequence $\{f^{m(i_{j})}_{i_{j}}\}$ converge to
$\bf{0}$. Let $K$ be a compact subset of $X$ and $U=<$ $\bf{0}$ $,
K, (-1,1)>$ be a base neighborhood of $\bf{0}$, then there exists
$j_0\in \mathbb{N}$ such that $f^{m(i_{j})}_{i_{j}}\in U$ for each
$j>j_0$. It follows that $K\subset U^{m(i_{j})}_{i_{j}}$ for
$j>j_0$. We thus get $X$ satisfies
$S_{fin}(\Gamma^{sh}_k,\Gamma_k)$. By Lemma \ref{lem2},
$S_{fin}(\Gamma^{sh}_k,\Gamma_k)=S_{1}(\Gamma^{sh}_k,\Gamma_k)$.

$(4)\Leftrightarrow(3)$. By Theorem \ref{th07}.

$(3)\Rightarrow(2)$ is immediate.

$(2)\Rightarrow(1)$.  For each $n\in \mathbb{N}$, let $S_n$ be a
sequentially dense subset of $C_k(X)$ and let $\{h_n: n\in
\mathbb{N}\}$ be sequentially dense in $C_k(X)$. Take a sequence
$\{f^m_n: m\in \mathbb{N}\}\subset S_n$ such that $f^m_n\mapsto
h_n$ ($m\mapsto \infty$). Then $f^m_n-h_n\mapsto \bf{0}$
($m\mapsto \infty$). Hence, there exist $f^{m_n}_n$ such that
$f^{m_n}_n-h_n\mapsto \bf{0}$ ($n\mapsto \infty$). We see that
$\{f^{m_n}_n: n\in \mathbb{N}\}$ is sequentially dense. Let $h\in
C_k(X)$ and take a sequence $\{h_{n_j}: j\in \mathbb{N}\}\subset
\{h_n:n\in \mathbb{N}\}$ converging to $h$. Then,
$f^{m_{n_j}}_{n_j}=(f^{m_{n_j}}_{n_j}-h_{n_j})+h_{n_j}\mapsto h$
($j\mapsto \infty$).

$(4)\Leftrightarrow(8)$. By Theorem \ref{th07}.

The proofs of the remaining implications are similar to those
proved above.

\end{proof}

A well-known that if $X$ is hemicompact then $C_k(X)$ is
metrizable. It follows that $C_k(X)$ is sequential separable for a
hemicompact space $X$ with $iw(X)=\aleph_0$. But, for general
case, author does not know the answer to the next question.

\medskip

{\bf Question 1.} Characterize a Tychonoff space $X$ such that a
space $C_k(X)$ is sequential separable (i.e.,  what is the
condition $X\models SS$ ?)

\medskip

\begin{proposition} (Corollary 4.8 (Dow-Barman) in \cite{bbm2}) Every
Fr$\acute{e}$chet-Urysohn separable $T_2$ space is selectively
separable (hence, selectively sequentially separable).
\end{proposition}

A well-known that a Tychonoff space $X$ the space $C_k(X)$ is
Fr$\acute{e}$chet-Urysohn if and only if  $X$ satisfies
$S_{1}(\mathcal{K},\Gamma_k)$ (\cite{llt}).

\medskip

Recall that the $i$-weight $iw(X)$ of a space $X$ is the smallest
infinite cardinal number $\tau$ such that $X$ can be mapped by a
one-to-one continuous mapping onto a Tychonoff space of the weight
not greater than $\tau$.

A Tychonoff space $X$ the space $C_{k}(X)$ is separable iff
$iw(X)=\aleph_0$ \cite{nob}.

\medskip

{\bf Question 2.} Is there a Tychonoff space $X$ with
$iw(X)=\aleph_0$ such that $C_k(X)$ satisfies
$S_{1}(\mathcal{S},\mathcal{S})$, but $C_k(X)$ is not
Fr$\acute{e}$chet-Urysohn (i.e. $X$ satisfies
$S_{1}(\Gamma^{sh}_k,\Gamma_k)$, but it has not property
$S_{1}(\mathcal{K},\Gamma_k)$) ?


\bigskip
\bibliographystyle{model1a-num-names}
\bibliography{<your-bib-database>}







\end{document}